\documentclass[11pt,a4paper]{article}

\usepackage{array}
\usepackage{multirow}

\makeatletter
\newcommand{\thickhline}{%
    \noalign {\ifnum 0=`}\fi \hrule height 1pt
    \futurelet \reserved@a \@xhline
}
\newcolumntype{"}{@{\hskip\tabcolsep\vrule width 1pt\hskip\tabcolsep}}
\makeatother

\usepackage{graphics}
\usepackage{epsfig}
\usepackage{subfigure}
\usepackage{amssymb}
\usepackage{amsthm}
\usepackage{mathrsfs}
\usepackage{hhline}
\usepackage{longtable}
\usepackage{amsmath}
\usepackage{float}
\usepackage{picinpar}
\usepackage{cite}
\usepackage{enumerate}
\usepackage{xcolor}

\usepackage[mathlines]{lineno}
\modulolinenumbers[2]
\newcommand*\patchAmsMathEnvironmentForLineno[1]{%
\expandafter\let\csname old#1\expandafter\endcsname\csname #1\endcsname  \expandafter\let\csname oldend#1\expandafter\endcsname\csname end#1\endcsname  \renewenvironment{#1}%
{\linenomath\csname old#1\endcsname}%
{\csname oldend#1\endcsname\endlinenomath}}%
\newcommand*\patchBothAmsMathEnvironmentsForLineno[1]{%
\patchAmsMathEnvironmentForLineno{#1}%
\patchAmsMathEnvironmentForLineno{#1*}}%
\AtBeginDocument{%
\patchBothAmsMathEnvironmentsForLineno{equation}%
\patchBothAmsMathEnvironmentsForLineno{align}%
\patchBothAmsMathEnvironmentsForLineno{flalign}%
\patchBothAmsMathEnvironmentsForLineno{alignat}%
\patchBothAmsMathEnvironmentsForLineno{gather}%
\patchBothAmsMathEnvironmentsForLineno{multline}%
}

\def\blue{\color{blue}}
\def\red{\color{red}}

\def\cyan{\color{cyan}}

\def\violet{\color{violet}}

\newtheorem{theorem}{Theorem}[section]
\newtheorem{lemma}[theorem]{Lemma}
\newtheorem{example}{Example}[section]
\newtheorem{corollary}[theorem]{Corollary}
\newtheorem{problem}{Problem}[section]

\newtheorem{conjecture}{Conjecture}[section]

\numberwithin{equation}{section}
\def\Z{\mathbb{Z}}

\newtheorem{defi}{Definition}[section]

\usepackage{graphicx}
\graphicspath{%
    {converted_graphics/}
    {/}
}

\def\blue{\color{blue}}
\def\red{\color{red}}

\def\cyan{\color{cyan}}

\def\violet{\color{violet}}

\def\Z{\mathbb{Z}}

\begin{document}
\baselineskip18truept
\normalsize
\begin{center}
{\mathversion{bold}\Large \bf On number of pendants in local antimagic chromatic number}

\bigskip
{\large  Gee-Choon Lau{$^{a,}$}\footnote{Corresponding author.}, Wai-Chee Shiu{$^{b,c}$}, Ho-Kuen Ng{$^d$}}\\

\medskip

\emph{{$^a$}Faculty of Computer \& Mathematical Sciences,}\\
\emph{Universiti Teknologi MARA (Segamat Campus),}\\
\emph{85000, Johor, Malaysia.}\\
\emph{geeclau@yahoo.com}\\

\medskip

\emph{{$^b$}College of Global Talents, Beijing Institute of Technology,\\ Zhuhai, China.}\\
\emph{{$^c$}Department of Mathematics, The Chinese University of Hong Kong,}\\
\emph{Shatin, Hong Kong.}\\

\emph{wcshiu@associate.hkbu.edu.hk}\\

\medskip

\emph{{$^d$}Department of Mathematics, San Jos\'{e} State University,}\\
\emph{San Jos\'e CA 95192 USA.}\\
\emph{ho-kuen.ng@sjsu.edu}\\
\end{center}


\medskip
\begin{abstract}
An edge labeling of a connected graph $G = (V, E)$ is said to be local antimagic if it is a bijection $f:E \to\{1,\ldots ,|E|\}$ such that for any pair of adjacent vertices $x$ and $y$, $f^+(x)\not= f^+(y)$, where the induced vertex label $f^+(x)= \sum f(e)$, with $e$ ranging over all the edges incident to $x$.  The local antimagic chromatic number of $G$, denoted by $\chi_{la}(G)$, is the minimum number of distinct induced vertex labels over all local antimagic labelings of $G$. Let $\chi(G)$ be the chromatic number of $G$. In this paper, sharp upper and lower bounds of $\chi_{la}(G)$ for $G$ with pendant vertices, and sufficient conditions for the bounds to equal, are obtained.  Consequently, for $k\ge 1$, there are infinitely many graphs with $k \ge \chi(G) - 1$ pendant vertices and $\chi_{la}(G) = k+1$. We conjecture that every tree $T_k$, other than certain caterpillars, spiders and lobsters, with $k\ge 1$ pendant vertices has $\chi_{la}(T_k) = k+1$.

\medskip
\noindent Keywords: Local antimagic labeling, local antimagic chromatic number, cut-vertices, pendants.
\medskip

\noindent 2010 AMS Subject Classifications: 05C78, 05C69.
\end{abstract}

\tolerance=10000
\baselineskip12truept
\def\qed{\hspace*{\fill}$\Box$\medskip}

\def\s{\,\,\,}
\def\ss{\smallskip}
\def\ms{\medskip}
\def\bs{\bigskip}
\def\c{\centerline}
\def\nt{\noindent}
\def\ul{\underline}
\def\lc{\lceil}
\def\rc{\rceil}
\def\lf{\lfloor}
\def\rf{\rfloor}
\def\a{\alpha}
\def\b{\beta}
\def\n{\nu}
\def\o{\omega}
\def\ov{\over}
\def\m{\mu}
\def\t{\tau}
\def\th{\theta}
\def\k{\kappa}
\def\l{\lambda}
\def\L{\Lambda}
\def\g{\gamma}
\def\d{\delta}
\def\D{\Delta}
\def\e{\epsilon}
\def\lg{\langle}
\def\rg{\tongle}
\def\p{\prime}
\def\sg{\sigma}
\def\to{\rightarrow}

\newcommand{\K}{K\lower0.2cm\hbox{4}\ }
\newcommand{\cl}{\centerline}
\newcommand{\om}{\omega}
\newcommand{\ben}{\begin{enumerate}}

\newcommand{\een}{\end{enumerate}}
\newcommand{\bit}{\begin{itemize}}
\newcommand{\eit}{\end{itemize}}
\newcommand{\bea}{\begin{eqnarray*}}
\newcommand{\eea}{\end{eqnarray*}}
\newcommand{\bear}{\begin{eqnarray}}
\newcommand{\eear}{\end{eqnarray}}

\section{Introduction}

\nt A connected graph $G = (V, E)$ is said to be {\it local antimagic} if it admits a {\it local antimagic edge labeling}, i.e., a bijection $f : E \to \{1,\dots ,|E|\}$ such that the induced vertex labeling $f^+ : V \to \Z$ given by $f^+(u) = \sum f(e)$ (with $e$ ranging over all the edges incident to $u$) has the property that any two adjacent vertices have distinct induced vertex labels~\cite{Arumugam}. Thus, $f^+$ is a coloring of $G$. Clearly, the order of $G$ must be at least 3.  The vertex label $f^+(u)$ is called the {\it induced color} of $u$ under $f$ (the {\it color} of $u$, for short, if no ambiguous occurs). The number of distinct induced colors under $f$ is denoted by $c(f)$, and is called the {\it color number} of $f$. The {\it local antimagic chromatic number} of $G$, denoted by $\chi_{la}(G)$, is $\min\{c(f) : f\mbox{ is a local antimagic labeling of } G\}$. Clearly, $2\le \chi_{la}(G)\le |V(G)|$. Haslegrave\cite{Haslegrave} proved that every graph is local antimagic. This paper relates the number of pendant vertices of $G$ to $\chi_{la}(G)$. Sharp upper and lower bounds, and sufficient conditions for the bounds to equal, are obtained. Consequently, there exist infinitely many graphs with $k\ge \chi(G)-1 \ge 1$ pendant vertices and $\chi_{la}(G)=k+1$.  We conjecture that every tree $T_k$, other than certain caterpillars, with $k\ge 1$ pendant vertices has $\chi_{la}(T_k) = k+1$. The following two results in~\cite{LSN2} are needed. 

\begin{lemma}\label{lem-pendant} Let $G$ be a graph of size $q$ containing $k$ pendants. Let $f$ be a local antimagic labeling of $G$ such that $f(e)=q$. If $e$ is not a pendant edge, then $c(f)\ge k+2$.\end{lemma}

\begin{theorem}\label{thm-pendant}  Let $G$ be a graph having $k$ pendants. If $G$ is not $K_2$, then $\chi_{la}(G)\ge k+1$ and the bound is sharp.\end{theorem}

\nt The sharp bound for $k\ge 2$ is given by the star $S_k$ with maximum degree $k$. The left labeling below is another example for $k=2$. The right labeling shows that the lower bound is sharp for $k=1$.\\\\
\centerline{\epsfig{file=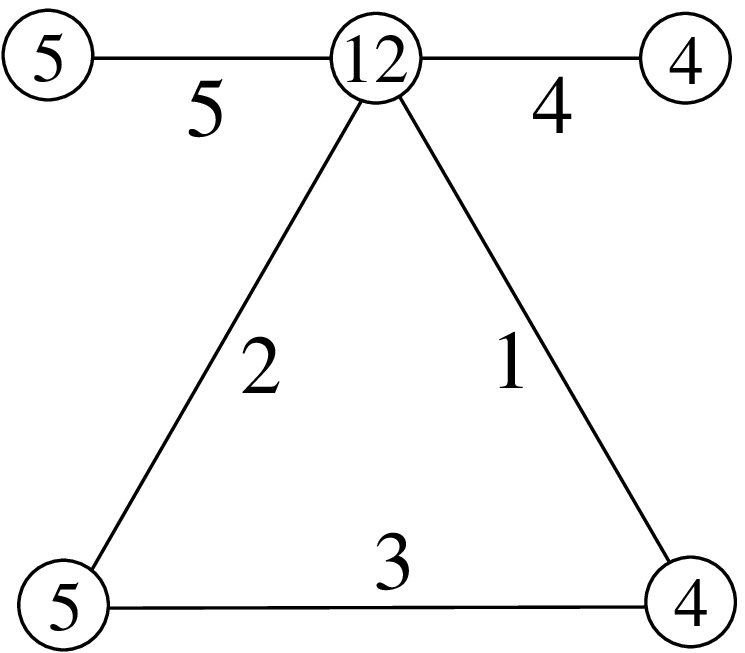, width=2.5cm}\qquad \epsfig{file=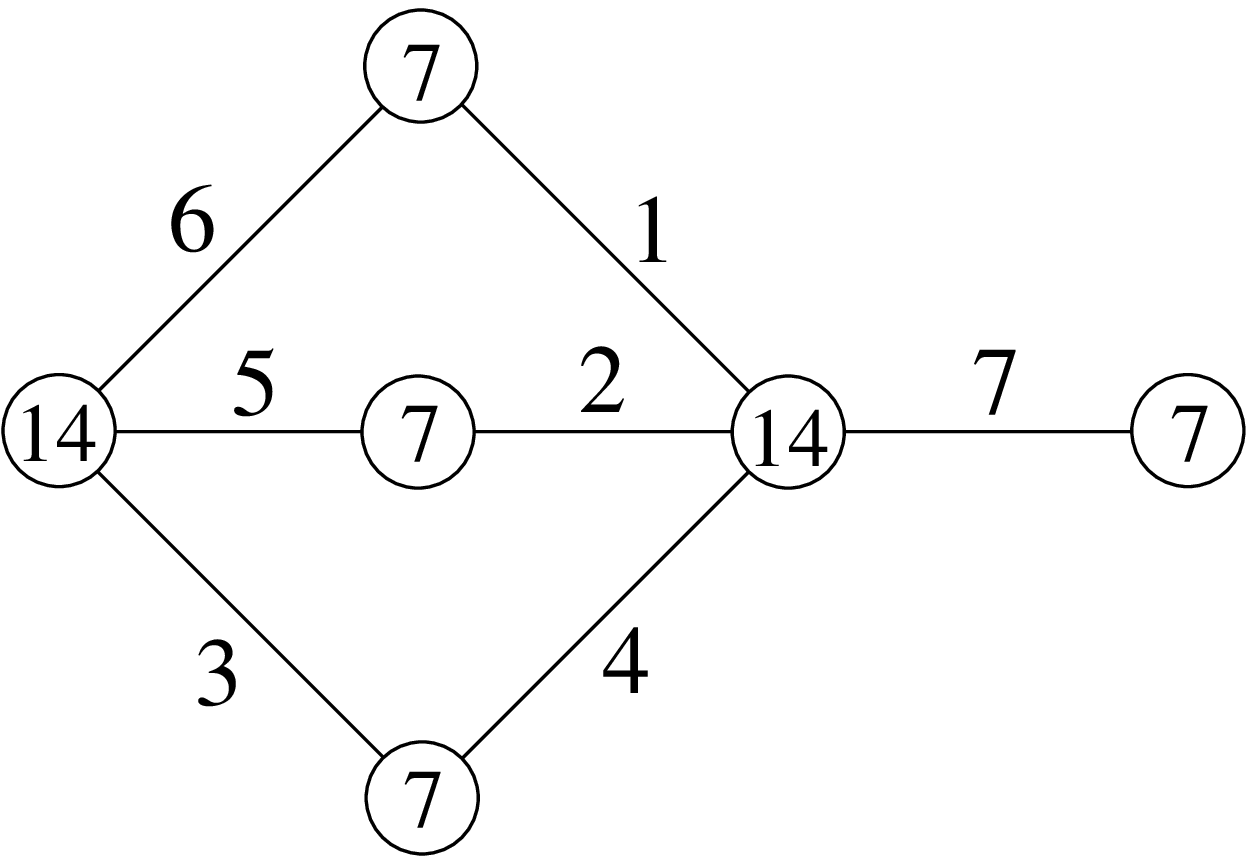, width=3.3cm}}\\

\nt For $ 1\le i\le t, a_i, n_i\ge 1$ and $d=\sum^t_{i=1} n_i\ge 3$, a spider of $d$ legs, denoted $Sp(a_1^{[n_1]}, a_2^{[n_2]}, \ldots, a_t^{[n_t]})$, is a tree formed by identifying an end-vertex of $n_i$ path(s) of length $a_i$. The vertex $u$ of degree $d$ is the core of the spider. Note that $Sp(1^{[n]})$ is the star graph of $n$ pendant vertices with $\chi_{la}(Sp(1^{[n]})) = n+1$. We first give a family of spiders with $k$ pendant vertices to have $\chi_{la} > k+1$.

\begin{theorem} For $n\ge 3$, $$\chi_{la}(Sp(2^{[n]})=\begin{cases}n+2 &\mbox{ if } n \ge 4 \\ n+1 &\mbox{ otherwise.} \end{cases}$$
\end{theorem}

\begin{proof} Let the neighbors of $u$ be $v_1, \dots, v_n$. Let $w_i$ be the pendant adjacent to $v_i$, $1\le i\le n$. By Theorem~\ref{thm-pendant}, it is easy to conclude that $\chi_{la}(Sp(2^{[3]})) = 4$. Consider $n\ge 4$. Let $f$ be a local antimagic labeling of $Sp(2^{[n]})$ with minimum $c(f)$. By Lemma~\ref{lem-pendant} and symmetry, suffice to consider $f(v_1w_1)=2n = f^+(w_1)$. By definition, $f^+(u)\ne f^+(v_1)$. Now, $f^+(u) > n(n+1)/2 > 2n \ge f^+(w_i)$ and $f^+(v_1) = 2n > f^+(w_i)$ for $1\le i\le n$. Thus, $c(f)\ge n+2$ and $\chi_{la}(Sp(2^{[n]}))\ge n+2$.  \\

\nt Define a bijection $f : E(Sp(2^{[n]}))\to [1,2n]$ such that $f(uv_i) = i$ and $f(v_iw_i) = 2n + 1 - i$. We have $f^+(u) = n(n+1)/2 > f^+(v_i) = 2n+1 > f^+(w_i) = 2n + 1 - i$, $1\le i\le n$. Consequently, $\chi_{la}(Sp(2^{[n]}))= n+2$ for $n\ge 4$. The theorem holds.
\end{proof}

\section{Adding pendant edges}

Suppose $G$ has $e\ge 2$ edges with $\chi_{la}(G) = t\ge 2$ such that the corresponding local antimagic labeling $f$ induces a $t$-independent set $\{V_1, V_2, \ldots, V_t\}$ with $|V_i|=n_i\ge 1$. Moreover, each non-pendant vertex must be in one of $V_i$ for $1\le i\le r\le t$, and $V_i$ is a singleton consisting of a pendant vertex for $r+1\le i\le t$ if $r < t$. Let $f^+(v) = c_i$ for each vertex $v\in V_i$. Without loss of generality, we assume that $c_1 < c_2 <\cdots < c_r$ and that $c_{r+1} < c_{r+2} < \cdots < c_t$. Note that $c_{r+1}$ to $c_t$ do not exist for $r=t$. By Theorem~\ref{thm-pendant}, $G$ contains at most $t-1$ pendant vertices. \\

Let $b\ge 0$ be the number of pendant vertices in $\cup^{r}_{i=1} V_i$ so that $G$ has exactly $t-r+b$ pendant vertices. Let $f(u'v')=e \le f^+(u') < f^+(v')$ such that $v'\in V_i, 1\le i\le r$ and that $u'\in V_j$. Suppose $e= f^+(u')$, then $u'$ is a pendant vertex. We now have (a) $j < i \le r$ so that $1\le b\le j$ or (b) $i \le r < j$ so that $0\le b \le i-1$.  Otherwise, if $e < f^+(u')$, then $u'$ is not a pendant vertex so that $j < i \le r $ and $0\le b\le j-1$. \\


 Let $G(V_i,s)$ be a graph obtained from $G$ by adding $s$ pendant edges $v_aw_{a,k}$ $(1\le k\le s)$ to each vertex in $V_i=\{v_a\,|\,1\le a\le n_i\}$.\\


\begin{theorem}\label{thm-addpendant} Suppose $G$ and $G(V_i,s)$ are as defined above. Let $r\ge 2$ and
$$e+sn_i\ge \begin{cases}c_r &\mbox{ for } i<r\\ c_{r-1} &\mbox{ for } i = r \end{cases}$$
such that $s\ge 1$ if $n_i=1$, and $s\ge 2$ is even if $n_i\ge 2$.

\begin{enumerate}[(1)]
  \item Suppose $e < c_1$. For $1\le i\le r$, $\chi_{la}(G(V_i,s)) = sn_i + t-r + 1$.
  \item Suppose $c_1 \le e < c_2$, then $\chi_{la}(G(V_1,s)) = sn_1 + t-r + 1$.  For $2\le i \le r$, $sn_i + t-r + b + 1\le \chi_{la}(G(V_i,s)) \le sn_i + t-r + 2$. Moreover, if $c_1=e$ and $b=1$, then $\chi_{la}(G(V_i,s)) = sn_i + t-r + 2$.
  \item Suppose $c_{j-1} \le e < c_{j}$ for $3\le j\le r$, then
  \begin{enumerate}[(a)]
    \item $sn_i + t-r+b \le \chi_{la}(G(V_i,s)) \le sn_i + t-r+j-1$ for $1\le i\le j-1$, if $V_i$ has a pendant vertex;
    \item $sn_i + t-r+b+1 \le \chi_{la}(G(V_i,s)) \le sn_i + t-r+j-1$ for $1\le i\le j-1$, if $V_i$ has no pendant vertex;
    \item $sn_i + t-r+b+1\le \chi_{la}(G(V_i,s)) \le sn_i + t-r+j$ for $j\le i\le r$.
  \end{enumerate}
Moreover, when $c_{j-1} = e$ and $b=j-1$, $\chi_{la}(G(V_i,s)) = sn_i + t-r+j-1$ for $1\le i\le j-1$, and $\chi_{la}(G(V_i,s)) = sn_i + t-r+j$ for $j\le i\le r$.
\end{enumerate}
In particular, if $c_{r-1}\le e < c_r$, then in $G(V_r,s)$, the condition on $e+sn_i$ is simplified to $s\ge 1$ for $n_r=1$, and $s\ge 2$ is even for $n_r\ge 2$.
\end{theorem}

\begin{proof}  Note that $G$ must contain a non-pendant vertex $v$ such that $f^+(v)>e$. So $c_r > e\ge c_t$.  Moreover, $uv\not\in E(G)$ implies that $uv\not\in E(G(V_i,s))$.  For $1\le i\le r$, define a bijection $g_i: E(G(V_i,s))\to [1,e+sn_i]$ such that $g_i(e) = f(e)$ if $e\in E(G)$ and that
$$g_i(v_aw_{a,k})= \textstyle e+[k+\frac{(-1)^k}{2}-\frac{1}{2}]n_i -(-1)^k(a-\frac{1}{2})+\frac{1}{2}=\begin{cases}e+(k-1)n_i+a &\mbox{ if } k \mbox{ is odd,}\\ e+kn_i +1-a & \mbox{ if } k \mbox{ is even.}\end{cases}$$
Therefore $g_i^+(v_a) = c_i + es + \frac{s}{2}(sn_i+1)> e+sn_i$ when $s$ is even; $g_i^+(v_a) = c_i + es + \frac{s}{2}(s+1)+a-1> e+sn_i$ when $s$ is odd (in this case $n_i=1$ and $a=1$).

\ms\nt Observe that $g_i^+(v) = f^+(v)$ for each $v\not\in V_i$ so that $c_1, \ldots, c_{i-1}, c_{i+1}, \ldots, c_t$ are vertex colors under $g_i$. Moreover, for $v_a\in V_i$, $f^+(v_a)$ is replaced by $g_i^+(v_a)>c_t$ for $1\le a\le n_i$ and\break $\{g_i^+(w_{a,k}) = g_i(v_aw_{a,k})\;|\;1\le a\le n_i, 1\le k\le s\} = [e+1, e+sn_i]$.

\begin{enumerate}[(1)]
  \item Since $e<c_1$, in $G(V_i,s)$, we have $c_{r+1} < \cdots <c_t < e+1 \le c_1 < \cdots < c_r \le e+sn_i$ when $1\le i<r$, and $c_{r+1} < \cdots <c_t < e+1 \le c_1 < \cdots < c_{r-1} \le e+sn_i$ when $i=r$. So $(\{c_j\;|\; 1\le j\le r\}\setminus\{c_i\}) \subset [e+1, e+sn_i]$ but all the $c_{r+1}, \ldots, c_t$ are not in $[e+1, e+sn_i]$.
      Since $c_i + es + \frac{s}{2}(n_is+1)> e+sn_i$, $g_i$ is a local antimagic labeling with induced vertex color set $\{c_{r+1}, c_{r+2}, \ldots, c_t\}\cup [e+1, e+sn_i] \cup \{c_i + es + \frac{s}{2}(n_is+1)\}$. Therefore, $\chi_{la}(G(V_i,s)) \le sn_i + t-r + 1$. Since $G(V_i,s)$ contains at least $sn_i+t-r$ pendant vertices, by Theorem~\ref{thm-pendant}, $\chi_{la}(G(V_i,s)) \ge sn_i + t-r + 1$. Thus, $\chi_{la}(G(V_i,s)) = sn_i + t-r + 1$.

  \item For $c_1\le e < c_2$, in $G(V_i,s)$, similar to the above case, $(\{c_j\;|\; 1\le j\le r\}\setminus\{c_1,c_i\})  \subset [e+1, e+sn_i]$ but all the $c_1, c_{r+1}, \ldots, c_t$ are not in $[e+1, e+sn_i]$.

      Suppose $i=1$, then $g_1$ is a local antimagic labeling with induced vertex color set $\{c_{r+1}, c_{r+2}, \ldots, c_t\}\cup [e+1, e+sn_1] \cup \{c_1 + es + \frac{s}{2}(n_1s+1)\}$.
      Thus, $\chi_{la}(G(V_1,s)) \le sn_1 + t-r + 1$. By the same argument of (1), we get $\chi_{la}(G(V_1,s))= sn_1 + t-r + 1$.

      Suppose $2\le i\le r$. In this case, $0\le b\le 1$. Similar to the above case, $g_i$ is a local antimagic labeling with induced vertex color set $\{c_{r+1}, c_{r+2}, \ldots, c_t, c_1\}\cup [e+1, e+sn_i] \cup \{c_i + es + \frac{s}{2}(sn_i+1)\}$. Thus, $\chi_{la}(G(V_1,s)) \le sn_i + t-r + 2$. Combining with Theorem~\ref{thm-pendant}, we have $sn_i + t-r +b+ 1\le \chi_{la}(G(V_i,s)) \le sn_i + t-r + 2$. Moreover, $e=c_1$ implies that $b=1$ so $\chi_{la}(G(V_i,s)) = sn_i + t-r + 2$.

  \item For $c_{j-1} \le e < c_j$, $3\le j\le r$, we have $0\le b\le j-1$. In $G(V_i,s)$, similar to above case, $\{c_j, \ldots, c_r\} \subset [e+1, e+sn_i]$ but all the $c_1, \ldots, c_{j-1}$ are not in $[e+1, e+sn_i]$. Similar to the above case, if $1\le i\le j-1$, then $g_i$ is a local antimagic labeling with induced vertex color set $\{c_k\;|\; r+1\le k\le t\}\cup(\{c_k\;|\; 1\le k\le j-1\}\setminus\{c_i\})\cup [e+1, e+sn_i] \cup \{c_i + es + \frac{s}{2}(n_is+1)\}$ so that $\chi_{la}(G(V_i,s)) \le sn_i + t-r + j-1$. If $j\le i\le r$, then the induced vertex color set is $\{c_k\;|\; r+1\le k\le t\}\cup\{c_k\;|\; 1\le k\le j-1\}\cup [e+1, e+sn_i] \cup \{c_i + es + \frac{s}{2}(n_is+1)\}$ so that $\chi_{la}(G(V_i,s)) \le sn_i + t-r + j$.

 Note that for $1\le i\le j-1$, if $V_i$ has a pendant vertex, then $G(V_i,s)$ has $sn_i+t-r+b-1$ pendant vertices, otherwise $G(V_i,s)$ has $sn_i+t-r+b$ pendant vertices. For $j\le i\le t$, $G(V_i,s)$ also has $sn_i+t-r+b$ pendant vertices. Combining with Theorem~\ref{thm-pendant}, we have the conclusion. Moreover, if $c_{j-1}=e$ and $b=j-1$, only case (a) exists for $1\le i\le j-1$ so that $\chi_{la}(G(V_i,s))=sn_i+t-r+j-1$, and $\chi_{la}(G(V_i,s))=sn_i+t-r+j$ for $j\le i\le r$.
\end{enumerate}
In particular, if $c_{r-1}\le e < c_r$, we know the condition $e+sn_r \ge c_{r-1}$ always hold and can be omitted.
\end{proof}


\begin{example}\label{eg-wheel}\hspace*{\fill}{}
\begin{enumerate}[(1)]
  \item In~\cite[Theorem 5]{LNS}, the authors proved that every $G=W_{4k}, k\ge 1$ with $e=8k$ edges admits a local antimagic labeling with $\chi_{la}(W_{4k})=3$ such that for $k\ge 2$, $e < c_1=9k+2 < c_2 = 11k+1 < c_3 = 2k(12k+1)$, while $W_4$ has $c_1=11,c_2=15,c_3=20$. Moreover, $n_1=n_2=2k, n_3=1$. Thus, $r=t=3$. Suppose $k\ge 2$, we can add $s \ge 12k-3$ ($s$ even) pendant edges to each vertex in $V_i, i=1$ or $2$, and label them accordingly. We can also add $s\ge 3k+1$ pendant edges to the vertex in $V_3$ and label them accordingly. By Theorem~\ref{thm-addpendant}, we get $W_{4k}(V_i,s)$ that has $\chi_{la} = sn_i + 1$ for the respective $s$ and $n_i$. We can also add edges to $W_4$ similarly.
  \item The right graph under Theorem~\ref{thm-pendant}, say $G$, has $e=7$ with $t=r=2$; $c_1=7, c_2=14$ and $n_1=4, n_2=2$. Thus, $\chi_{la}(G(V_1,s)) = 4s+1$. In particular, since $c_1=e$ and $b=1$, we also can get $\chi_{la}(G(V_2,s)) = 2s + 2$.  
  \item The left graph under Theorem~\ref{thm-pendant}, say $H$, has $e=5$ with $t=r=3$; $c_1=4 < c_2 = e < c_3$ and $b=2$. Thus, $\chi_{la}H(V_3,s) = s + 3$ for $s\ge 1$.
\end{enumerate}

\end{example}

\nt By a similar argument for Theorem~\ref{thm-addpendant}, we can get the following theorem.

\begin{theorem} \label{thm-addpendant2} Suppose $G\not\cong K_{1,e}$ and $G(V_i,s)$ are as defined above with $r < t$ and $e + s \ge c_r$. Let $r+1\le i\le t$.
\begin{enumerate}[(1)]
 \item If $e < c_1$, then $\chi_{la}(G(V_i,s)) = s + t-r$.
 \item If $c_{j-1} \le e < c_j$ for $2\le j\le r$, then $s+t-r+b\le \chi_{la}(G(V_i,s))\le s + t-r + j-1$. Moreover, if $b=j-1$, then $\chi_{la}(G(V_i,s)) = s + t-r + j-1$.  
\end{enumerate}
\end{theorem}

\begin{example}\hspace*{\fill}{}
\begin{enumerate}[(1)]
  \item  It is easy to verify that the graph $G$ below has $\chi_{la}(G) = 9$ with $e=9 < c_1=10, c_2=20, c_3=25$ and $r=3 < t=9$. We may apply Theorem~\ref{thm-addpendant2} (1) accordingly.
\vskip3mm
\centerline{\epsfig{file=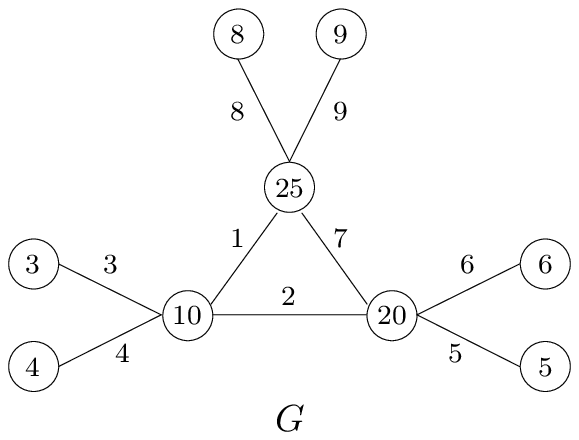, width=3cm}}
  \item We can add $k\ge 1$ pendant edges to the degree 4 vertex of the left graph under Theorem~\ref{thm-pendant} and label the edges by $6$ to $5+k$ bijectively to get a graph, say $G$, with $k+2$ pendant vertices. Clearly, $r=3 < t = k+3$ with $c_1=4, c_2=5 < e= k+5 < c_3 =(k+3)(k+8)/2$, $c_i = i+2$ for $i=4, \ldots, k+2$, $j=3$ and $b=j-1$. By Theorem~\ref{thm-addpendant2} (2), we know $\chi_{la}(G(V_i,s)) = s+k+2$ for $s\ge (k+2)(k+7)/2$.
\end{enumerate}
\end{example}

\nt Suppose $G$ is a graph containing $k\ge 1$ pendant vertices with $\chi_{la}(G)=k+1$. Keeping the notation defined before Theorem~\ref{thm-addpendant}, we have $t=k+1$. Then there is only one independent set, which is $V_r$, containing no pendant vertex. Clearly, $r\ge 1$. So we have $c_1<\cdots <c_{r-1}\le e<c_r$ and $c_{r+1}<\cdots <c_{k+1} \le e$ for $r\ge 2$, whereas for $r=1$, $G\cong K_{1,k}, k\ge 2$ with $c_1 = k(k+1)/2$ and $c_i = i-1$ for $2 \le i\le k+1$.

\begin{corollary}\label{cor-addpendant3} Keep all notation defined before Theorem~\ref{thm-addpendant}.
Suppose $G$ has $k$ pendant vertices and $\chi_{la}(G)=k+1$. Let
$$e+sn_i\ge \begin{cases}c_r &\mbox{ for } i\in[1, k+1]\setminus \{r\}\\ c_{r-1} &\mbox{ for } i = r \end{cases}$$
such that $s\ge 1$ if $n_i=1$, and $s\ge 2$ is even if $n_i\ge 2$,
 then $G(V_r,s)$ has $sn_r+k-1$ pendant vertices with $\chi_{la}(G(V_r,s)) = sn_r+k+1$, whereas $G(V_i,s)$ has $sn_i+k-1$ pendant vertices with $\chi_{la}(G(V_i,s))=sn_i+k$ when $i\in[1, k+1]\setminus \{r\}$.
\end{corollary}


\begin{proof} Suppose $r=1$. Recall that $e=k$ and $c_1=k(k+1)/2$. Since $G(V_1,s)\cong K_{1,k+s}$, we only consider $G(V_i,s), 2\le i\le k+1$. By labeling the $s$ added edges of $G(V_i,s)$ by $k+1$ to $k+s$ bijectively, $c_i$ is now replaced by $i-1+ks + s(s+1)/2 > k+s \ge c_1 > c_{k+1}>\cdots > c_2$. Thus, $G(V_i,s)$ now admits a local antimagic labeling with vertex color set $[1,k+s]\setminus \{i-1\} \cup \{i-1+ks + s(s+1)/2\}$. Therefore, $\chi_{la}(G(V_i,s)) \le s+k$. Since $G(V_i,s)$ has $s+k-1$ pendant vertices, by Theorem~\ref{thm-pendant}, the equality holds.

\ms\nt Consider $r\ge 2$. Suppose $i=r$, then $G(V_r,s)$ has $sn_r + k$ pendant vertices so that $\chi_{la}(G(V_r,s))\ge sn_r+k+1$. By a labeling $g_i$ as in the proof of Theorem~\ref{thm-addpendant}, we know $g_i$ is a local antimagic labeling with induced vertex color set $\{c_1,c_2,\ldots, c_{r-1}, c_{r+1}, \ldots, c_{k+1}\}\cup [e+1,e+sn_r]\cup\{g_r^+(v) \,|\, v \in V_r\}$ of size $sn_r+k+1$. By Theorem~\ref{thm-pendant}, we have $\chi_{la}(G(V_r,s))= sn_r+k+1$.

\ms\nt Suppose $i\in[1,k+1]\setminus\{r\}$, then $G(V_i,s)$ has $sn_i+k-1$ pendant vertices so that $\chi_{la}(G(V_i,s))\ge sn_i+k$. By a labeling $g_r$ as in the proof of Theorem~\ref{thm-addpendant}, we know $g_r$ is a local antimagic labeling with induced vertex color set $(\{c_1,c_2, \ldots, c_{k+1}\}\setminus\{c_i,c_r\})\cup [e+1,e+sn_i]\cup\{g_i^+(v) \,|\, v \in V_i\}$ of size $sn_r+k$. By Theorem~\ref{thm-pendant}, we have $\chi_{la}(G(V_r,s))= sn_r+k$.
\end{proof}

\begin{example} Consider $W_4$ under Example~\ref{eg-wheel}. We have $W_4(V_3,12)$ with $k=12$, $\chi_{la}(W_4(V_3,12)) = 13$ and $c_1=11, c_2=15 < e=20 < c_3=194$. Moreover, $P_n, n\ge 3$ and $K_{1,k}$, is a tree with $k\ge 2$ pendant vertices and $\chi_{la} = k+1$. We may apply Corollary~\ref{cor-addpendant3} accordingly.
\end{example}

\nt Thus, we get the following.


\begin{theorem} For $k\ge 1$, there exist infinitely many graphs $G$ with $k \ge \chi(G)-1 \ge 1$ pendant vertices and $\chi_{la}(G) = k+1$. \end{theorem}

\nt Let $G-e$ be the graph $G$ with an edge $e$ deleted. In~\cite[Lemmas 2.2-2.4]{LSN}, the authors obtained sufficient conditions for $\chi_{la}(G) = \chi_{la}(G-e)$. We note that these lemmas may be applied to $G(V_i,s)$ if all vertices in each $V_j, 1\le j\le t$, are of same degree like $W_{4k}(V_i,s)$ in Example~\ref{eg-wheel}. \\

\section{Existing results and open problems}


\nt In\cite{Arumugam+L+P+W},~~\cite[Theorem 7 and Theoreom 9]{LNS} and~\cite[Theorems 2.4-2.6, 2.9, 3.1, Lemma 2.10 and Section 3]{LSN2}, the authors also determined the exact value of $\chi_{la}(G)$ for many families of $G$ with pendant vertices. Particularly, they showed that there are infinitely many caterpillars of $k$-pendant vertices with $\chi_{la} = k+1$ or $\chi_{la} \ge k+2$. Note that~\cite[Theorem~2.6]{LSN2} corrected Theorem~2.2 in~\cite{Nuris+S+D}. A lobster is a tree such that the removal of its pendant vertices resulted in a caterpillar. Note that the graph $Sp(2^{[n]})$ is also a lobster. We end this paper with the followings.

\begin{conjecture} Every tree $T_k$, other than certain caterpillars, spiders and lobsters, with $k\ge 2$ pendant vertices has $\chi_{la}(T_k) = k+1$. \end{conjecture}

\begin{problem} Characterize all graphs $G$ with $k\ge \chi(G) - 1 \ge 1$ pendant vertice(s) and $\chi_{la}(G) = k+1$. \end{problem}

\end{document}